\documentclass[12pt]{amsart}
\usepackage{amssymb,amsmath,amsthm,latexsym,mathtools,dsfont}
\usepackage{mathrsfs}
\usepackage{mdwlist}
\usepackage{graphicx}
\usepackage{enumerate}
\usepackage[shortlabels]{enumitem}
\usepackage{a4wide}

\usepackage{ifthen}
\usepackage{xargs}

\makeatletter
\newcommand\@dotsep{4.5}
\def\@tocline#1#2#3#4#5#6#7{\relax
  \ifnum #1>\c@tocdepth 
  \else
    \par \addpenalty\@secpenalty\addvspace{#2}%
    \begingroup \hyphenpenalty\@M
    \@ifempty{#4}{%
      \@tempdima\csname r@tocindent\number#1\endcsname\relax
    }{%
      \@tempdima#4\relax
    }%
    \parindent\z@ \leftskip#3\relax
    \advance\leftskip\@tempdima\relax
    \rightskip\@pnumwidth plus1em \parfillskip-\@pnumwidth
    #5\leavevmode\hskip-\@tempdima #6\relax
    \leaders\hbox{$\m@th
      \mkern \@dotsep mu\hbox{.}\mkern \@dotsep mu$}\hfill
    \hbox to\@pnumwidth{\@tocpagenum{#7}}\par
    \nobreak
    \endgroup
  \fi}
\makeatother 

\let\oldtocsection=\tocsection
\let\oldtocsubsection=\tocsubsection

\renewcommand{\tocsection}[2]{\hspace{0em}\oldtocsection{#1}{#2}}
\renewcommand{\tocsubsection}[2]{\hspace{22pt}\oldtocsubsection{#1}{#2}}

\newcommandx{\yaHelper}[2][1=\empty]{%
\ifthenelse{\equal{#1}{\empty}}%
  { \ensuremath{ \scriptstyle{ #2 } } } 
  { \raisebox{ #1 }[0pt][0pt]{ \ensuremath{ \scriptstyle{ #2 } } } }  
}   
\newcommandx{\yrightarrow}[4][1=\empty, 2=\empty, 4=\empty, usedefault=@]{%
  \ifthenelse{\equal{#2}{\empty}}
  { \xrightarrow{ \protect{ \yaHelper[ #4 ]{ #3 } } } } 
  { \xrightarrow[ \protect{ \yaHelper[ #2 ]{ #1 } } ]{ \protect{ \yaHelper[ #4 ]{ #3 } } } } 
}

\input xy
\xyoption{all}
\usepackage{physics}

\usepackage{xcolor}
\definecolor{darkgreen}{RGB}{0, 153, 51}
\definecolor{violet}{RGB}{112, 73, 170}
\definecolor{darkred}{RGB}{153, 0, 0}

\usepackage{tikz}
\usetikzlibrary{cd}
\usetikzlibrary{calc}

\newcommand{\N}{\mathbb{N}}

\newcommand{\Aa}{\mathcal{A}}
\newcommand{\Ba}{\mathcal{B}}

\newcommand{\FF}{\mathcal{F}}

\newcommand{\LL}{\mathscr{L}}

\newcommand{\e}{\varepsilon}
\newcommand{\w}{\widetilde}

\newcommand{\n}[1]{\|#1\|}
\newcommand{\nn}[1]{{\vert\kern-0.25ex\vert\kern-0.25ex\vert #1 
    \vert\kern-0.25ex\vert\kern-0.25ex\vert}}
\newcommand{\lnn}[1]{{\left\vert\kern-0.25ex\left\vert\kern-0.25ex\left\vert #1 
    \right\vert\kern-0.25ex\right\vert\kern-0.25ex\right\vert}}

\renewcommand{\leq}{\leqslant}
\renewcommand{\geq}{\geqslant}

\newcommand{\cs}{{\rm C}$^\ast$}

\newcommand{\mdef}{\mathsf{m}\mbox{-}\mathrm{def}}
\newcommand{\adef}{\mathsf{a}\mbox{-}\mathrm{def}}

\newcommand{\usim}{\,\raise.17ex\hbox{$\scriptstyle\mathtt{\sim}$}}

\newtheorem{theorem}{Theorem}
\newtheorem{lemma}[theorem]{Lemma}
\newtheorem{proposition}[theorem]{Proposition}
\newtheorem{corollary}[theorem]{Corollary}

\theoremstyle{definition}

\theoremstyle{remark}
\newtheorem*{remark}{Remark}

\numberwithin{equation}{section}

\title[Almost multiplicative maps]{Almost multiplicative maps with respect to almost associative operations on Banach algebras}

\author{Tomasz Kochanek}
\address{Institute of Mathematics, University of Warsaw, Banacha~2, 02-097 Warsaw, Poland}
\email{tkoch@mimuw.edu.pl}

\thanks{This work has been supported by the National Science Centre grant no. 2020/37/B/ST1/01052.}

\begin{document}
\maketitle

\begin{center}
{\it Dedicated to Professors Maciej Sablik and L\'aszl\'o Sz\'ekelyhidi\\ on the occasion of their 70$^{\,th}$ birthday.}
\end{center}

\begin{abstract}
We prove an Ulam type stability result for a~non-associative version of the multiplicativity equation, that is, $T(xy)=\Psi(T(x),T(y))$, where $T$ is a~unital bounded operator acting from an~amenable Banach algebra to a~dual Banach algebra.
\end{abstract}
\maketitle

\section{Introduction}
\noindent
Let $\Aa$ and $\Ba$ be Banach algebras, and let $\LL^n(\Aa,\Ba)$ be the space of $n$-linear bounded operators from $\Aa^n$ to $\Ba$. In this note, we study an~Ulam type stability problem for a~non-associative version of the multiplicativity equation, namely, 
\begin{equation}\label{E}
T(xy)=\Psi(T(x),T(y))\qquad (x,y\in\Aa),
\end{equation}
where $T\in\LL^1(\Aa,\Ba)$ and $\Psi\in\LL^2(\Ba,\Ba)$. An~expected stability effect should be quite different than the one for almost multiplicative maps. Indeed, $\Psi$ being merely `almost' (but not exactly) associative on the range of $T$ means that \eqref{E} cannot be satisfied accurately, since the left-hand side yields an~associative operation on $(x,y)$. Hence, the best accuracy with which equation \eqref{E} can be satisfied should be expressed in terms of the `defect' of $\Psi$ concerning its associativity.

Our main result says, roughly speaking, that for certain Banach algebras, any operator $T$ satisfying \eqref{E} with sufficient accuracy can be approximated to a~prescribed accuracy by an~operator for which the obvious obstacle of non-associativity of $\Psi$ is actually the only obstacle on the way of satisfying \eqref{E}.

To be more precise, for any linear map $T\colon\Aa\to\Ba$ and any bilinear map $\Psi\colon\Ba\times\Ba\to\Ba$, consider $T^{\,\vee}\colon\Aa^2\to\Ba$ defined by
$$
T^{\,\vee}(x,y)=T(xy)-\Psi(T(x),T(y)).
$$
We define the $\Psi$-{\it multiplicative defect} of $T$ as $\n{T^{\,\vee}}$, that is,
$$
\mdef_\Psi(T)=\sup\big\{\n{T(xy)-\Psi(T(x),T(y))}\colon x,y\in\Aa,\, \n{x}, \n{y}\leq 1\big\}.
$$
Similarly, we define the {\it associative defect} of $\Psi$ by 
$$
\adef(\Psi)=\sup\big\{\n{\Psi(u,\Psi(v,w))-\Psi(\Psi(u,v),w)}\colon u,v,w\in\Ba,\, \n{u}, \n{v}, \n{w}\leq 1\big\}.
$$
Our main result then reads as follows.
\begin{theorem}\label{main_T}
Let $\Aa$ be a unital amenable Banach algebra and $\Ba$ be a~unital dual Banach algebra with an~isometric predual. Then for arbitrary $K,L\geq 1$ and $\e,\eta\in (0,1)$ there exists $\delta>0$ such that the following holds true: If $\Psi\in\LL^2(\Ba,\Ba)$ satisfies $\n{\Psi}\leq L$ and $\Psi(1,u)=\Psi(u,1)=u$ for every $u\in\Ba$, and $T\in\LL^1(\Aa,\Ba)$ is a~unital operator with $\n{T}\leq K$ and $\mdef_\Psi(T)\leq\delta$, then there exists a~unital $S\in\LL^1(\Aa,\Ba)$ such that
$$
\mdef_\Psi(S)<\adef(\Psi)^{1-\eta}\quad\mbox{ and }\quad\n{T-S}<\e.
$$
\end{theorem}

\vspace*{1mm}
In order to put this result in wider context, recall that the theory of almost multiplicative maps between Banach algebras was developed by B.E.~Johnson in the~series of papers \cite{johnson_memoir}, \cite{johnson_d}, \cite{johnson_f}, \cite{johnson}, where he introduced the AMNM {\it property}. A~pair $(\Aa,\Ba)$ of Banach algebras has this property, provided that for any $K,\e>0$ there is $\delta>0$ such that for every $T\in\LL^1(\Aa,\Ba)$ satisfying $\n{T(xy)-T(x)T(y)}\leq\delta$ there exists a~multiplicative operator $S\in\LL^1(\Aa,\Ba)$ with $\n{T-S}<\e$. One of the brilliant ideas of Johnson was to define a~`Banach-algebraic' analogue of the Newton--Raphson approximation procedure which allowed him to show that if $\Aa$ is amenable and $\Ba$ is a~dual Banach algebra, then $(\Aa,\Ba)$ has the AMNM property (see \cite[Thm.~3.1]{johnson}). His result was generalized in the PhD thesis of Horv\'ath \cite{bence}, and widely developed in the~recent paper \cite{CHL} by Choi, Horv\'ath and Laustsen, where they exhibited a~large collection of AMNM pairs of algebras of bounded operators on Banach spaces. Many other Ulam type stability problems on Banach algebras were studied in recent years; see e.g. \cite{alaminos}, \cite{BOT}, \cite{choi}, \cite{kochanek}, \cite{KV}.

The proof of Theorem~\ref{main_T} closely follows the above-mentioned method of proving \cite[Thm.~3.1]{johnson} by using a~Newton--Raphson-like algorithm defined in terms of approximate diagonal. However, unlike in the case of Johnson's method, unless $\Psi$ is associative, our approximation process must terminate at certain point.

\section{Terminology and preparatory lemmas}

\noindent
Throughout this section, we fix unital Banach algebras $\Aa$ and $\Ba$, as well as operators $T\in\LL^1(\Aa,\Ba)$ and $\Psi\in\LL^2(\Ba,\Ba)$.

One of the main ideas of approximating almost multiplicative maps (see \cite{johnson} and \cite{CHL}) is based on modification of standard coboundary operators in the Hochschild cohomology theory, as defined in \cite[Ch.~2]{runde}, by introducing as a~parameter a~map $\phi\in\LL(\Aa,\Ba)$ playing the role of `approximate action' of $\Aa$ on $\Ba$. According to \cite[Def.~5.2]{CHL}, for any $n\in\N_0$, the $n$-{\it coboundary operator} $\partial^n_\phi\colon \LL^n(\Aa,\Ba)\to\LL^{n+1}(\Aa,\Ba)$ is defined by the formula
\begin{equation*}
\begin{split}    \partial^n_\phi\psi(a_1,\ldots,a_{n+1})=\phi(a_1)\psi(a_2,\ldots,a_{n+1}) +\sum_{j=1}^n (-1)^j &\psi(a_1,\ldots,a_ja_{j+1},\ldots,a_{n+1})\\[-1pt]
    &+(-1)^{n+1}\psi(a_1,\ldots,a_n)\psi(a_{n+1}).
\end{split}
\end{equation*}
In fact, in the study of almost multiplicative maps only the $2$-coboundary operator is relevant. For studying equation \eqref{E}, we introduce the following modification of $\partial_\phi^2$. Define $\delta_T^2\colon \LL^2(\Aa,\Ba)\to\LL^3(\Aa,\Ba)$ by the formula
$$
\delta^2_T\,\phi(x,y,z)=\Psi(T(x),\phi(y,z))-\phi(xy,z)+\phi(x,yz)-\Psi(\phi(x,y),T(z)).
$$
Observe that if $\Psi$ is `almost' associative, then the operator $T^{\,\vee}$ `almost' satisfies the relation $T^{\,\vee}\in\mathrm{ker}\,\delta_T^2$ in the following sense.
\begin{lemma}\label{almost_ker_L}
For any $T\in\LL^1(\Aa,\Ba)$ and $\Psi\in\LL^2(\Ba,\Ba)$, we have
$$
\n{\delta_T^2\, T^{\,\vee}}\leq\adef(\Psi)\cdot\n{T}^3.
$$
\end{lemma}
\begin{proof}
Note that
\begin{equation*}
\begin{split}
[\delta_T^2\, T^{\,\vee}](x,y,z) &= \Psi(T(x),T^{\,\vee}(y,z))+\Psi(T(xy),T(z))\\
&\hspace*{64pt}-\Psi(T(x),T(yz))-\Psi(T^{\,\vee}(x,y),T(z))\\
&=\Psi[\Psi(T(x),T(y)),T(z)]-\Psi[T(x),\Psi(T(y),T(z))],
\end{split}
\end{equation*}
from which the result follows readily.
\end{proof}

\begin{lemma}\label{derivative_L}
The Fr\'echet derivative of the map
$$
\LL^1(\Aa,\Ba)\ni S\xmapsto[\phantom{xx}]{}S^{\,\vee}\in\LL^2(\Aa,\Ba)
$$
at the point $T$ is given by the formula
$$
[\mathsf{D}H](x,y)=H(xy)-\Psi(T(x),H(y))-\Psi(H(x),T(y))\qquad (H\in\LL^1(\Aa,\Ba)).
$$
\end{lemma}
\begin{proof}
It follows by a straightforward calculation:
\begin{equation*}
\begin{split}
[(T+H)^{\vee}-T^{\,\vee}](x,y) &=H(xy)-\Psi(T(x)+H(x),T(y)+H(y))+\Psi(T(x),T(y))\\
&=H(xy)-\Psi(T(x),H(y))-\Psi(H(x),T(y))-r(H)(x,y),
\end{split}
\end{equation*}
where $r(H)(x,y)=\Psi(H(x),H(y))$ and hence
\begin{equation}\label{R_norm}
\n{r(H)}=\n{(T+H)^{\,\vee}-T^{\,\vee}-\mathsf{D}H}\leq \n{\Psi}\cdot\n{H}^2.\qedhere
\end{equation}
\end{proof}

In what follows, we use standard notation and facts concerning projective tensor products of Banach spaces, as described e.g. in \cite[Ch.~2]{ryan}. By $\Aa\hat{\otimes}\Aa$ we denote the completion of $\Aa\otimes\Aa$ in the projective norm. Recall that for every $\Phi\in\LL^2(\Aa,\Ba)$ there exists a~unique operator $\Theta\in\LL^1(\Aa\hat{\otimes}\Aa,\Ba)$ such that $\Theta(a\otimes b)=\Phi(a,b)$ for all $a,b\in\Aa$ and, moreover,  $\n{\Phi}=\n{\Theta}$. From now on, we adopt the convention that for any $\Phi\in\LL^2(\Aa,\Ba)$, we write $\w\Phi$ for the corresponding linear operator from $\Aa\hat{\otimes}\Aa$ to $\Ba$. Let also $\pi_\Aa\colon \Aa\hat{\otimes}\Aa\to\Aa$ stand for the unique bounded linear map such that $\pi_\Aa(a\otimes b)=ab$ for all $a,b\in\Aa$.

Let us now recall some elementary facts concerning amenable Banach algebras; for a~nice exposition of this theory, see \cite{runde} and \cite{runde_new}. A~net $(\Delta_\alpha)_{\alpha\in A}\subset \Aa\hat{\otimes}\Aa$ is called an~{\it approximate diagonal}, provided that 
$$
\lim_\alpha(a\cdot\Delta_\alpha-\Delta_\alpha\cdot a)=0\quad\mbox{ and }\quad \lim_\alpha a\pi_\Aa(\Delta_\alpha)=a\quad\,\,(a\in\Aa).
$$
Note that it the case $\Aa$ is unital the latter condition means that $\lim_\alpha\pi_\Aa(\Delta_\alpha)=1$ in norm. The approximate diagonal $(\Delta_\alpha)_{\alpha\in A}$ is called {\it bounded} if $\sup_\alpha\n{\Delta}_\alpha<\infty$. Notice that, in our convention, $\pi_\Aa=\w\Lambda$, where $\Lambda(a,b)=ab$. Therefore, $\n{\pi_\Aa}=\n{\Lambda}=1$ and since $\lim_\alpha\pi_\Aa(\Delta_\alpha)=1$, we must have $\sup_\alpha\n{\Delta_\alpha}\geq 1$.

By a {\it virtual diagonal} for $\Aa$ we mean an~element $\mathsf{M}\in (\Aa\hat{\otimes}\Aa)^{\ast\ast}$ such that 
$$
a\cdot \mathsf{M}=\mathsf{M}\cdot a\quad\mbox{ and }\quad a\cdot \pi_\Aa^{\ast\ast}\mathsf{M}=a\quad (a\in\Aa).
$$
A~Banach algebra $\Aa$ is called {\it amenable} if there exists an~approximate diagonal $(\Delta_\alpha)_{\alpha\in A}$ for $\Aa$, and this is equivalent to saying that there is a~virtual diagonal $\mathsf{M}$ for $\Aa$; see \cite[Thm.~2.2.4]{runde}. Moreover, we can always assume that $\mathsf{M}$ is the weak$^\ast$ limit in $(\Aa\hat{\otimes}\Aa)^{\ast\ast}$ of the net $(\Delta_\alpha)_{\alpha\in A}$.

Concerning the codomain algebra in our stability problem, we need the following definition: A~Banach algebra $\Ba$ is called a~{\it dual Banach algebra} with {\it isometric predual} $X$, provided that $X$ is a~Banach space such that $\Ba$ and $X^\ast$ are isometrically isomorphic and multiplication in $\Ba$ is separately $\sigma(\Ba,X)$-continuous. In particular, $\Ba$ is equipped with the weak$^\ast$ topology and the condition of the continuity of multiplication is equivalent to saying that $i^\ast\circ\kappa(X)$ is a~sub-$\Ba$-bimodule of $\Ba^\ast$, where $i\colon \Ba\to X^\ast$ is the said linear isometry and $\kappa\colon X\to X^{\ast\ast}$ is the canonical embedding (see \cite[\S 2]{daws2}).

It should be remarked that the original assumption made by Johnson was that the codomain algebra $\Ba$ is isomorphic, as a~Banach $\Ba$-bimodule, to the dual $(\Ba_\ast)^\ast$ of some Banach $\Ba$-bimodule $\Ba_\ast$. Our choice follows the definition proposed in \cite{CHL} which was influenced by \cite{daws1} and \cite{daws2}.

\section{Proof of the main result}
\noindent
From now on, we fix Banach algebras $\Aa$ and $\Ba$ satisfying the assumptions of Theorem~\ref{main_T}. Let $(\Delta_\alpha)_{\alpha\in A}\subset \Aa\hat{\otimes}\Aa$ be a~bounded approximate diagonal for $\Aa$ with
$$
\Delta_\alpha=\sum_j a_{\alpha,j}\otimes b_{\alpha,j}\qquad (\alpha\in A).
$$
All these are finite sums yet we do not need to indicate the sets over which we sum. The following observation is due to Johnson (see also \cite[Lemma~5.11]{CHL}); we include a~short proof for the sake of completeness.
\begin{lemma}\label{correct_L}
For any $R\in\LL^2(\Aa,\Ba)$, the limit
\begin{equation*}
\lim_\alpha \sum_j R(a_{\alpha,j},b_{\alpha,j})
\end{equation*}
exists in the weak$^\ast$ topology $\sigma(\Ba,\Ba_\ast)$ on $\Ba$.
\end{lemma}
\begin{proof}
Let $\kappa\colon\Ba_\ast\to\Ba^{\ast}$ be the canonical embedding, so that $\kappa^\ast\colon\Ba^{\ast\ast}\to\Ba$ is an~isometric projection. Let also $\mathsf{M}\in (\Aa\hat{\otimes}\Aa)^{\ast\ast}$ be the virtual diagonal for $\Aa$ such that $\lim_\alpha\Delta_\alpha=\mathsf{M}$ in the weak$^\ast$ topology. Consider the corresponding linear operator $\w{R}\in\LL^1(\Aa\hat{\otimes}\Aa,\Ba)$ and its second adjoint $\w{R}^{\ast\ast}\in\LL^1((\Aa\hat{\otimes}\Aa)^{\ast\ast},\Ba^{\ast\ast})$. We have
$$
\lim_\alpha\sum_j R(a_{\alpha,j},b_{\alpha,j})=\lim_\alpha\w{R}(\Delta_\alpha)=\lim_\alpha \kappa^\ast \w{R}^{\ast\ast}(\Delta_\alpha)=\kappa^\ast\w{R}^{\ast\ast}(\mathsf{M}).
$$
\end{proof}

Similarly as in Johnson's method (see the proof of \cite[Thm.~3.1]{johnson}), we define an~operator $\mathsf{J}\colon \LL^2(\Aa,\Ba)\to\LL^1(\Aa,\Ba)$ by
\begin{equation}\label{J}
\mathsf{J}R(x)=\lim_\alpha \sum_j \Psi(T(a_{\alpha,j}),R(b_{\alpha,j},x)).
\end{equation}
This definition is correct in view of Lemma~\ref{correct_L}. Note also that for any weak$^\ast$ convergent net $(y_i)_{i\in I}\subset\Ba$ we have $\n{\!\lim_i y_i}\leq \limsup_i\n{y_i}$. We will use this observation several times in the sequel.

We shall show that formula \eqref{J} yields an~approximate right inverse of the derivative $\mathsf{D}$ given in Lemma~\ref{derivative_L} on the set of operators `almost' belonging to $\mathrm{ker}\,\delta_T^2$. In order to simplify writing let us denote $u\circ v=\Psi(u,v)$.
\begin{proposition}\label{inverse_P}
Assume that $T$ is unital and $\Psi(1,u)=u$ for each $u\in\Ba$. For every $R\in\LL^2(\Aa,\Ba)$, we have
\begin{equation}\label{RDJR}
\begin{split}
\n{R+\mathsf{D}\mathsf{J}R}\leq \big( &2\!\cdot\!\mdef_\Psi(T)\cdot\!\n{\Psi}\!\cdot\!\n{R}\\
&+\n{\psi}\!\cdot\!\n{T}\!\cdot\!\n{\delta_T^2 R}+\adef(\Psi)\!\cdot\!\n{T}^2\!\cdot\!\n{R}\big)\!\cdot\!\sup_\alpha\n{\Delta_\alpha}.
\end{split}
\end{equation}
\end{proposition}
\begin{proof}
First, observe that since $\lim_\alpha\pi_\Aa(\Delta_\alpha)=1$, we have $\lim_\alpha\sum_j a_{\alpha,j}b_{\alpha,j}=1$ in norm and hence $\lim_\alpha \sum_j T(a_{\alpha,j}b_{\alpha,j})=1$. For any $x,y\in\Aa$, using the identity $\Psi(1,u)=u$ and Lemma~\ref{derivative_L}, we thus obtain

\vspace*{2mm}
\begin{equation}\label{R+DJR1}
\begin{split}
[R+\mathsf{D}\mathsf{J}R](x,y) &=R(x,y)+[\mathsf{J}R](xy)-T(x)\circ [\mathsf{J}R](y)-[\mathsf{J}R](x)\circ T(y)\\[4pt]
&=\lim_\alpha \sum_j\Big\{T(a_{\alpha,j}b_{\alpha,j})\circ R(x,y)+T(a_{\alpha,j})\circ R(b_{\alpha,j},xy)\\[-8pt]
& \hspace*{124pt}-T(x)\circ\big(T(a_{\alpha,j})\circ R(b_{\alpha,j},y)\big)\\
& \hspace*{124pt}-\big(T(a_{\alpha,j})\circ R(b_{\alpha,j},x)\big)\circ T(y)\Big\}.
\end{split}
\end{equation}
Since 
$$
R(b_{\alpha,j},xy)=\delta_T^2\,R(b_{\alpha,j},x,y)-T(b_{\alpha,j})\circ R(x,y)+R(b_{\alpha,j}x,y)+R(b_{\alpha,j},x)\circ T(y),
$$
we have 
\begin{equation*}
\begin{split}
\Big\|\lim_\alpha\sum_j\Big\{T( &a_{\alpha,j})\circ 
 R(b_{\alpha,j},xy)-T(a_{\alpha,j})\circ R(b_{\alpha,j}x,y)\\[-8pt]
&+T(a_{\alpha,j})\circ \big(T(b_{\alpha,j})\circ R(x,y)\big)-T(a_{\alpha,j})\circ\big(R(b_{\alpha,j},x)\circ T(y)\big)\Big\}\Big\|\\[6pt]
&\hspace*{140pt}\leq \n{\Psi}\!\cdot\!\n{T}\!\cdot\!\n{\delta_T^2 R}\!\cdot\!\n{x}\n{y}\sup_\alpha \n{\Delta_\alpha}.
\end{split}
\end{equation*}
This estimate follows by considering the bilinear operator $\Phi_1\colon\Aa\times\Aa\to\Ba$ given by
$$
\Phi_1(a,b)=T(a)\circ \delta_T^2 R(b,x,y)
$$
for which $\n{\Phi_1}\leq\n{\Psi}\!\cdot\!\n{T}\!\cdot\!\n{\delta_T^2 R}\!\cdot\!\n{x}\n{y}$, and for the corresponding $\w\Phi_1\colon\Aa\hat{\otimes}\Aa\to\Ba$ we have 
$$
\w\Phi_1(\Delta_\alpha)=\sum_j T(a_{\alpha,j})\circ \delta_T^2 R(b_{\alpha,j},x,y),
$$
thus 
$$
\Big\|\lim_\alpha \sum_j T(a_{\alpha,j})\circ \delta_T^2 R(b_{\alpha,j},x,y)\Big\|\leq \n{\Psi}\!\cdot\!\n{T}\!\cdot\!\n{\delta_T^2 R}\!\cdot\!\n{x}\n{y}\sup_\alpha\n{\Delta_\alpha}=:\!\eta_1(x,y).
$$
Coming back to \eqref{R+DJR1} we see that $[R+\mathsf{D}\mathsf{J}R](x,y)$ is at distance at most $\eta_1(x,y)$ from 
\begin{equation*}
\begin{split}
S(x,y)\coloneqq \lim_\alpha\sum_j \Big\{T( &a_{\alpha,j}b_{\alpha,j})\circ R(x,y)-T(a_{\alpha,j})\circ \big(T(b_{\alpha,j})\circ R(x,y)\big)\\[-8pt]
&+T(a_{\alpha,j})\circ\big(R(b_{\alpha,j},x)\circ T(y)\big)-\big(T(a_{\alpha,j})\circ R(b_{\alpha,j},x)\big)\circ T(y)\\[2pt]
&\hspace*{41pt}+T(a_{\alpha,j})\circ R(b_{\alpha,j}x,y)-T(x)\circ\big(T(a_{\alpha,j})\circ R(b_{\alpha,j},y)\big)\Big\}.
\end{split}
\end{equation*}
Considering the bilinear operator $\Phi_2\colon\Aa\times\Aa\to\Ba$ given by 
$$
\Phi_2(a,b)=T(a)\circ (T(b)\circ R(x,y))-(T(a)\circ T(b))\circ R(x,y),
$$
we have $\n{\Phi_2}\leq \adef(\Psi)\!\cdot\!\n{T}^2\!\cdot\!\n{R}\!\cdot\!\n{x}\n{y}$ and 
$$
\w\Phi_2(\Delta_\alpha)=\sum_j \Big\{T(a_{\alpha,j})\circ \big(T(b_{\alpha,j})\circ R(x,y)\big)-\big(T(a_{\alpha,j})\circ T(b_{\alpha,j})\big)\circ R(x,y)\Big\}.
$$
Hence, the difference in the first line of the definition of $S(x,y)$ can be estimated as follows:
\begin{equation}\label{est1}
\begin{split}
\Big\|\lim_\alpha\sum_j \Big\{ &T(a_{\alpha,j}b_{\alpha,j})\circ R(x,y)-T(a_{\alpha,j})\circ \big(T(b_{\alpha,j})\circ R(x,y)\big)\Big\}\Big\|\\
&\leq \Big\|\lim_\alpha\sum_j \Big\{\big(T(a_{\alpha,j}b_{\alpha,j})-T(a_{\alpha,j})\circ T(b_{\alpha,j})\big)\circ R(x,y)\Big\}\Big\|\\
&\hspace*{88pt}+\adef(\Psi)\!\cdot\!\n{T}^2\!\cdot\!\n{R}\!\cdot\!\n{x}\n{y}\sup_\alpha\n{\Delta_\alpha}=:\eta_2(x,y).
\end{split}
\end{equation}
Similarly, for the difference in the second line, we have
\begin{equation}\label{est2}
\Big\|\lim_\alpha\sum_j \Big\{T(a_{\alpha,j})\circ\big(R(b_{\alpha,j},x)\circ T(y)\big)-\big(T(a_{\alpha,j})\circ R(b_{\alpha,j},x)\big)\circ T(y)\Big\}\Big\|\leq\eta_2(x,y).
\end{equation}
Finally, in order to estimate the difference in the third line of the definition of $S(x,y)$, consider the bilinear map $\Phi_3\colon\Aa\times\Aa\to\Ba$ defined by $\Phi_3(a,b)=T(a)\circ R(b,y)$. By the diagonal property, $\lim_\alpha(x\cdot\Delta_\alpha-\Delta_\alpha\cdot x)=0$, we have 
$$
\lim_\alpha\w\Phi_3\Big(\sum_j a_{\alpha,j}\otimes b_{\alpha,j}x\Big)=\lim_\alpha\w\Phi_3\Big(\sum_j xa_{\alpha,j}\otimes b_{\alpha,j}\Big),
$$
which means that the norm of the difference in the third line defining $S(x,y)$ equals
\begin{equation}\label{est3}
\begin{split}
\Big\|\lim_\alpha\sum_j &\Big\{T(xa_{\alpha,j})\circ R(b_{\alpha,j},y)-T(x)\circ\big(T(a_{\alpha,j})\circ R(b_{\alpha,j},y)\big)\Big\}\Big\|\\
&\leq \Big\|\lim_\alpha\sum_j \Big\{T(xa_{\alpha,j})\circ R(b_{\alpha,j},y)-\big(T(x)\circ T(a_{\alpha,j})\big)\circ R(b_{\alpha,j},y)\Big\}\Big\|\\[-3pt]
&\hspace*{288pt}+\eta_2(x,y).
\end{split}
\end{equation}
Combining \eqref{est1}, \eqref{est2} and \eqref{est3} we obtain
\begin{equation*}
\begin{split}
\n{[R+\mathsf{D}\mathsf{J}R](x,y)} &\leq \Big\|\lim_\alpha\sum_j \Big\{\big(T(a_{\alpha,j}b_{\alpha,j})-T(a_{\alpha,j})\circ T(b_{\alpha,j})\big)\circ R(x,y)\Big\}\Big\|\\
&+\Big\|\lim_\alpha\sum_j \Big\{T(xa_{\alpha,j})\circ R(b_{\alpha,j},y)-\big(T(x)\circ T(a_{\alpha,j})\big)\circ R(b_{\alpha,j},y)\Big\}\Big\|\\[-3pt]
&\hspace*{226pt}+\eta_1(x,y)+3\eta_2(x,y)\\[5pt]
&\leq 2\!\cdot\!\mdef_\Psi(T)\cdot\!\n{\Psi}\!\cdot\!\n{R}\!\cdot\!\n{x}\n{y}\sup_\alpha\n{\Delta_\alpha}+\eta_1(x,y)+3\eta_2(x,y),
\end{split}
\end{equation*}
 which gives the desired estimate \eqref{RDJR}.
\end{proof}

\begin{corollary}\label{est_C}
Assume that $T$ is unital and $\Psi(1,u)=u$ for each $u\in\Ba$. Let also $(\Delta_\alpha)_{\alpha\in A}$ be a~bounded approximate diagonal in $\Aa\otimes\Aa$ with $M=\sup_\alpha\n{\Delta_\alpha}<\infty$. Then, the~operator $\mathsf{J}\colon\LL^2(\Aa,\Ba)\to\LL^1(\Aa,\Ba)$ defined by \eqref{J} satisfies 
\begin{equation}\label{cor_estimate}
\begin{split}
\n{T^{\,\vee}+\mathsf{D}\mathsf{J}\,T^{\,\vee}}\leq \big( &2\!\cdot\!(\mdef_\Psi(T))^2\n{\Psi}\\
&+\adef(\Psi)\!\cdot\!\mdef_\Psi(T)\!\cdot\!\n{T}^2+\adef(\Psi)\!\cdot\!\n{\psi}\!\cdot\!\n{T}^4\big) M.
\end{split}
\end{equation}
\end{corollary}
\begin{proof}
We simply apply Proposition~\ref{inverse_P} to $R=T^{\,\vee}$ and appeal to Lemma~\ref{almost_ker_L}.
\end{proof}

Following Johnson's idea, for any fixed bounded approximate diagonal $(\Delta_\alpha)_{\alpha\in A}\subset \Aa\otimes\Aa$, we define the `improving operator' $\FF\colon\LL^1(\Aa,\Ba)\to\LL^1(\Aa,\Ba)$ by
\begin{equation}\label{FT}
\FF\,T=T+\mathsf{J}T^{\,\vee},
\end{equation}
where $\mathsf{J}$ is given by formula \eqref{J}.

\begin{proposition}\label{defect_new_P}
Under the same assumptions as in {\rm Corollary~\ref{est_C}}, there exist constants $C_i=C_i(\n{\Psi},\n{T},M)\geq 0$, for $i=1,2,3$, such that
\begin{equation*}
\begin{split}
\mdef_\Psi(\FF T)\leq C_1( &\mdef_\Psi(T))^2\\
&+C_2\!\cdot\!\adef(\Psi)\!\cdot\!\mdef_\Psi(T)+C_3\!\cdot\!\adef(\Psi).
\end{split}
\end{equation*}
\end{proposition}
\begin{proof}
First, observe that considering the bilinear map $\Phi\colon\Aa\times\Aa\to\Ba$ given by $\Phi(a,b)=T(a)\circ T^{\,\vee}(b,x)$, for any fixed $x\in\Aa$, we have 
$$
\n{\Phi}\leq\mdef_\Psi(T)\!\cdot\! \n{\Psi}\!\cdot\!\n{T}\n{x},
$$
hence
\begin{equation}\label{JT_norm}
\n{\mathsf{J}T^{\,\vee}(x)}=\lim_\alpha\n{\w\Phi(\Delta_\alpha)}\leq \mdef_\Psi(T)\!\cdot\! \n{\Psi}\!\cdot\!\n{T}\!\cdot\! M\n{x}.
\end{equation}
Using \eqref{JT_norm} in combination with Lemma~\ref{derivative_L}, in particular estimate \eqref{R_norm}, as well as Corollary~\ref{est_C}, we get
\begin{equation*}
\begin{split}
\mdef_\Psi(\FF T) &=\n{(T+\mathsf{J}T^{\,\vee})^{\,\vee}}\\
&\leq \n{(T+\mathsf{J}T^{\,\vee})^{\,\vee}-T^{\,\vee}-\mathsf{D}\mathsf{J}T^{\,\vee}}+\n{T^{\,\vee}+\mathsf{D}\mathsf{J}T^{\,\vee}}\\
&\leq \n{\Psi}\!\cdot\!\n{\mathsf{J}T^{\,\vee}}^2+\n{T^{\,\vee}+\mathsf{D}\mathsf{J}T^{\,\vee}}\\
&\leq C_1 (\mdef_\Psi(T))^2+C_2\!\cdot\!\adef(\Psi)\!\cdot\!\mdef_\Psi(T)+C_3\!\cdot\!\adef(\Psi),
\end{split}
\end{equation*}
where
\begin{equation}\label{C1}
\left\{\begin{array}{rcl}
C_1(\n{\Psi},\n{T},M) & = & \n{\Psi}^3\!\cdot\!\n{T}^2\!\cdot\! M^2+2\n{\Psi}\!\cdot\! M\\
C_2(\n{\Psi},\n{T},M) & = & \n{T}^2\!\cdot\! M\\
C_3(\n{\Psi},\n{T},M) & = & \n{\Psi}\!\cdot\!\n{T}^4\!\cdot\! M.
\end{array}\right.
\end{equation}
\end{proof}

The proof of our main result relies on an~approximation procedure in which a~given operator $T_n$ is replaced by its `improved' version $T_n+\FF T_n$. We iterate this process as long as the obtained multiplicative defects are roughly larger than the associative defect of $\Psi$. This terminates at some point, unless we can continue and obtain an exact solution of equation \eqref{E}.

\begin{proof}[Proof of Theorem~\ref{main_T}]
Fix any $K,L\geq 1$, $\e, \theta\in (0,1)$, pick a~bounded approximate diagonal $(\Delta_\alpha)_{\alpha\in A}$ in $\Aa\otimes\Aa$ with $M=\sup_\alpha\n{\Delta_\alpha}<\infty$. Define $\delta\in (0,1)$ by the formula
\begin{equation}\label{delta_def}
\delta=\Big[2\big(2LM+e^{4LM}(L^3M^2+M)K^2+e^{8LM}LMK^4\big)\Big]^{-1/\theta}\mathbf{\e}.
\end{equation}

Fix any unital operator $T\in\LL^1(\Aa,\Ba)$ with $\n{T}\leq K$ and $\mdef_\Psi(T)\leq\delta$. Assume also that $\n{\Psi}\leq L$ and define $\alpha=\adef(\Psi)$. Let $\FF\colon\LL^1(\Aa,\Ba)\to\LL^1(\Aa,\Ba)$ be the `improving operator' defined by \eqref{FT}. We define a sequence $(T_n)_{n=0}^\infty\subset\LL^1(\Aa,\Ba)$ recursively by 
$$
T_0=T\quad\,\,\mbox{and}\,\,\quad T_n=\FF T_{n-1}\qquad\mbox{for }n\in\N.
$$
Suppose that $\alpha\leq (\mdef_\Psi(T))^{1+\theta}$. We define $N\in\N\cup\{\infty\}$ as the least natural number $n$ for which 
$$
\alpha>(\mdef_\Psi(T_n))^{1+\theta},
$$
provided that such an $n$ exists, and we set $N=\infty$ otherwise. Define $(\omega_n)_{n=0}^\infty$ recursively by $\omega_0=0$ and $\omega_n=1+(1+\theta)\omega_{n-1}$ for $n\in\N$. For any $n\in\N_0$, we also define
$$
\delta_n=2^{-\omega_n}\delta\quad\mbox{ and }\quad \beta_n=\prod_{j=0}^{n-1}(1+LM\delta_j).
$$
Observe that since $\sum_{j=0}^\infty\delta_j<\infty$, the product $\prod_{j=0}^\infty (1+LM\delta_j)$ converges, hence we have $\beta_n\nearrow B$ for some $B<\infty$. In fact, since $\omega_n\geq n$ for $n\in\N_0$, we have 
\begin{equation}\label{B_est}
B\leq\prod_{j=0}^\infty(1+2^{-j}LM\delta)\leq \exp\Big\{\sum_{j=0}^\infty 2^{-j}LM\Big\}=\exp(2LM).
\end{equation}
Set $K_n=\beta_n K$, so that we have $K_n<BK$ for each $n\in\N_0$.

\vspace*{2mm}\noindent
\underline{{\it Claim.}} For each integer $0\leq n<N$, the following conditions hold true:

\vspace*{1mm}
\begin{enumerate}[label={\rm (\roman*)}, leftmargin=28pt]
\setlength{\itemindent}{0pt}
\setlength{\itemsep}{3pt}

\item $T_n$ is unital;
\item $\mdef_\Psi(T_n)\leq\delta_n$;
\item $\n{T_n}\leq K_n$.
\end{enumerate}

\vspace*{2mm}\noindent
We proceed by induction. In order to show (i) notice that the identity $\Psi(u,1)=u$ implies $T^{\,\vee}(b,1)=0$ for every $b\in\Ba$. Hence, $\mathsf{J} T^{\,\vee}(1)=0$ which gives $T_1(1)=\FF T(1)=T(1)=1$ and, by induction, $T_n(1)=1$ for every $n\in\N$.

Assertions (ii) and (iii) are plainly true for $n=0$ as $\delta_0=\delta$ is not smaller than the defect of $T=T_0$, and $K_0=\beta_0K=K$. Fix $n\in\N_0$ with $n+1<N$ and suppose that estimates (ii) and (iii) hold true. Using Proposition~\ref{defect_new_P}, formulas \eqref{C1}, and the fact that $\alpha\leq (\mdef_\Psi(T_n))^{1+\theta}$, we obtain
\begin{equation*}
\begin{split}
\mdef_\Psi(T_{n+1}) &=\mdef_\Psi(\FF T_n)\\
&\leq C_1(\n{\Psi},\n{T_n},M)\delta_n^2+C_2(\n{\Psi},\n{T_n},M)\alpha\delta_n\\
&\hspace*{127pt}+C_3(\n{\Psi},\n{T_n},M)\alpha\\[1pt]
&\leq \big(C_1(\n{\Psi},\n{T_n},M)+C_2(\n{\Psi},\n{T_n},M)\\
&\hspace*{127pt}+C_3(\n{\Psi},\n{T_n},M)\big)\delta_n^{1+\theta}\\[1pt]
&\leq \big(2LM+K_n^2L^3M^2+K_n^2M+K_n^4LM\big)\delta_n^{1+\theta}\\[1pt]
&\leq\big(2LM+(L^3M^2+M)B^2K^2+B^4K^4LM\big)\delta_n^{1+\theta}
\end{split}
\end{equation*}
Therefore, appealing to the definition of $\delta$, that is, formula \eqref{delta_def}, and inequality \eqref{B_est}, we obtain
\begin{equation*}
\mdef_\Psi(T_{n+1})\leq \frac{1}{2}\delta^{-\theta}\delta_n^{1+\theta}=\frac{1}{2}\!\cdot\!2^{-(1+\theta)\omega_n}\delta=2^{-\omega_{n+1}}\delta=\delta_{n+1}.
\end{equation*}
By inequality \eqref{JT_norm}, we also obtain
\begin{equation*}
\begin{split}
\n{T_{n+1}} &=\n{\FF T_n}\leq \n{T_n}+\n{\mathsf{J}T_n^{\,\vee}}\\
&\leq K_n+\delta_n K_nLM\\
&=\beta_n(1+\delta_n LM)K=\beta_{n+1}K=K_{n+1},
\end{split}
\end{equation*}
which completes the induction.

Now, suppose that the inequality $(\mdef_\Psi(T_n))^{1+\theta}<\alpha$ is never satisfied, which means that our process does not terminate and we have $N=\infty$. Then, since for each $n\in\N$ we have
$$
\n{\mathsf{J}T_n^{\,\vee}}\leq \delta_n K_nLM<\delta_n BKLM\leq 2^{-n}\delta BKLM,
$$

\vspace*{1mm}\noindent
the sequence $(T_n)_{n=1}^\infty$ is a~Cauchy sequence, thus the norm limit $S=\lim_{n\to\infty}T_n$ exists. In view of (ii), we have $\lim_{n\to\infty}\mdef_\Psi(T_n)=0$ which yields $\mdef_\Psi(S)=0$. Moreover, $S-T=\sum_{n=0}^\infty \mathsf{J}T_n^{\,\vee}$ and hence
$$
\n{S-T}\leq\sum_{n=0}^\infty 2^{-n}\delta BKLM\leq 2\delta e^{2LM}KLM<\e,
$$
which means that in this case we have produced an~exact solution of the equation $S(xy)=S(x)\circ S(y)$ lying at distance smaller than $\e$ from $T$.

If our process terminates, i.e. $N<\infty$, then we have 
$$
\mdef_\Psi(T_N)<\alpha^{\frac{1}{1+\theta}}
$$
and, likewise in the previous case, 
$$
\n{T_N-T}\leq\sum_{n=0}^{N-1}\n{\mathsf{J}T_n^{\,\vee}}<\e.
$$
This completes the proof since for any $\eta\in (0,1)$ we can pick $\theta\in(0,1)$ so that $(1+\theta)^{-1}>1-\eta$ and then the parameter $\delta$ defined by \eqref{delta_def} does the job in the sense that whenever $T$ is as above, then the approximation $T_N$ satisfies $\mdef_\Psi(T_N)<\adef(\Psi)^{1-\eta}$. 
\end{proof}

\begin{remark}
It may be of some interest to ask if the power function $\alpha\mapsto\alpha^{1-\eta}$ appearing in Theorem~\ref{main_T} is optimal and, in general, how close to $\adef(\Psi)$ one can get with the multiplicative defect $\mdef_\Psi(S)$.
\end{remark}

\vspace*{2mm}\noindent
{\bf Acknowledgement. }The author acknowledges with gratitude the support from the National Science Centre, grant OPUS 19, project no.~2020/37/B/ST1/01052.

\bibliographystyle{amsplain}

\begin{thebibliography}{10}



\bibitem{alaminos} J. Alaminos, J.~Extremera, A.R.~Villena, \emph{Approximately zero-product-preserving maps}, Israel J.~Math.~{\bf 178} (2010), 1--28.





\bibitem{BOT} M. Burger, N. Ozawa, A.~Thom, \emph{On Ulam stability}, Israel J.~Math.~{\bf 193} (2013), 109--129.

\bibitem{choi} Y. Choi, \emph{Approximately multiplicative maps from weighted semilattice algebras}, J.~Aust. Math. Soc.~{\bf 95} (2013), 36--67.

\bibitem{CHL} Y. Choi, B.~Horv\'ath, N.J.~Laustsen, \emph{Approximately multiplicative maps between algebras of bounded operators on Banach spaces}, Trans. Amer. Math. Soc.~{\bf 375} (2022), 7121--7147.




\bibitem{daws1} M. Daws, \emph{Dual Banach algebras: representations and injectivity}, Studia Math.~{\bf 178} (2007), 231--275.

\bibitem{daws2} M. Daws, H. Le Pham, S. White, \emph{Conditions implying the uniqueness of the weak$^\ast$-topology on certain group algebras}, Houston J.~Math.~{\bf 35} (2009), 253--276.







\bibitem{bence} B. Horv\'ath, \emph{Algebras of operators on Banach spaces, and homomorphisms thereof}, PhD thesis, Department of Mathematics and Statistics, Lancaster University, 2019.

\bibitem{johnson_memoir} B.E. Johnson, \emph{Cohomology in Banach algebras}, Memoirs of the American Mathematical Society~127, American Mathematical Society, Providence, R.I.~1972.

\bibitem{johnson_d} B.E. Johnson, \emph{Approximate diagonals and cohomology of certain annihilator Banach algebras}, Amer. J.~Math.~{\bf 94} (1972), 685--698.

\bibitem{johnson_f} B.E. Johnson, \emph{Approximately multiplicative functionals}, J.~London Math. Soc.~{\bf 34} (1986), 489--510.

\bibitem{johnson} B.E. Johnson, \emph{Approximately multiplicative maps between Banach algebras}, J.~London Math. Soc.~{\bf 37} (1988), 294--316.

\bibitem{kochanek} T. Kochanek, \emph{Approximately order zero maps between \cs-algebras}, J.~Funct. Anal.~{\bf 281} (2021), Paper No. 109025, 49 pp. 

\bibitem{KV} P. McKenney, A. Vignati, \emph{Ulam stability for some classes of C$^\ast$-algebras}, Proc. Roy. Soc. Edinburgh Sect. A~{\bf 149} (2019), 45--59.


\bibitem{runde} V. Runde, \emph{Lectures on amenability}, Lecture Notes in Mathematics, Springer-Verlag, Berlin Heidelberg 2002.

\bibitem{runde_new} V. Runde, \emph{Amenable Banach algebras}, Springer Monographs in Mathematics, Springer-
Verlag, New York 2020.

\bibitem{ryan} R.A. Ryan, \emph{Introduction to tensor products of Banach spaces}, Springer-Verlag, London 2002.

\end{thebibliography}

\end{document}